\documentclass[pdflatex,sn-mathphys-num]{sn-jnl}

\usepackage{epsfig}
\usepackage{epstopdf}
\usepackage{graphicx}%
\usepackage{multirow}%
\usepackage{amsmath,amssymb,amsfonts}%
\usepackage{amsthm}%
\usepackage{mathrsfs}%
\usepackage[title]{appendix}%
\usepackage{xcolor}%
\usepackage{textcomp}%
\usepackage{manyfoot}%
\usepackage{booktabs}%
\usepackage{algorithm}%
\usepackage{algorithmicx}%
\usepackage{algpseudocode}%
\usepackage{listings}%

\theoremstyle{thmstyleone}%
\newtheorem{theorem}{Theorem}
\theoremstyle{thmstyletwo}%
\newtheorem{example}{Example}%
\newtheorem{remark}{Remark}%
\newtheorem{lemma}{Lemma}%

\theoremstyle{thmstylethree}%

\begin{document}

\title[\textbf{Fokker-Planck equation for a stochastic heat equation with $Q$-Wiener noise and non-homogeneous boundary conditions}]{\textbf{Fokker--Planck equation for stochastic heat equations}}

\author[1,2]{\fnm{Qingyan} \sur{ Meng}}\email{mengqy2024@gbu.edu.cn}

\author[1]{\fnm{Jinqiao} \sur{ Duan}}\email{duan@gbu.edu.cn}

\author*[3]{\fnm{Jinlong} \sur{ Wei}}\email{weijinlong.hust@gmail.com}

\author[4]{\fnm{Peter E.} \sur{ Kloeden}}\email{kloeden@math.uni-frankfurt.de}

\affil[1]{\orgdiv{Department of Mathematics \& Guangdong Provincial Key Laboratory of Mathematical and
Neural Dynamical Systems}, \orgname{Great Bay University}, \orgaddress{\city{Dongguan}, \postcode{523000},  \country{China}}}

\affil[2]{\orgdiv{School of Mathematical Sciences}, \orgname{University of Science and Technology of China}, \orgaddress{\city{Hefei}, \postcode{230026}, \country{China}}}

\affil[3]{\orgdiv{School of Statistics and Mathematics}, \orgname{Zhongnan University of Economics and Law}, \orgaddress{\city{Wuhan}, \postcode{430073}, \country{China}}}

\affil[4]{\orgdiv{Mathematisches Institut}, \orgname{Universit\"{a}t T\"{u}bingen}, \orgaddress{\city{T\"{u}bingen}, \postcode{72076}, \country{Germany}}}

\abstract{This work is devoted to the study of the Fokker--Planck equation for a stochastic heat equation with an additive $Q$-Wiener noise and non-homogeneous boundary conditions. We explicitly construct the probability density function and establish the associated Fokker--Planck equation by applying the eigenfunction expansion technique. Moreover, the Feynman--Kac formula is used to obtain the probabilistic representation of the solution. The analysis is further extended to cases with multiplicative noise involving nonlocal diffusion operators under homogeneous boundary conditions, as well as the corresponding Kardar--Parisi--Zhang (KPZ) equation. Notably, the evolution of the probability density function for the stochastic heat equation depends critically on the spatial location.}




\maketitle

\section{Introduction}\label{sec1}

Partial differential equations serve as a fundamental tool for depicting a multitude of physical phenomena, including continuous medium mechanics, electromagnetism, and heat conduction. In practical scenarios, however, numerous systems are susceptible to the influence of random factors, including the inherent uncertainty in quantum regimes, turbulence within fluids, and fluctuations in financial markets. Consequently, stochastic partial differential equations (SPDEs) are employed to address these complexities \cite{DaPrato,DW,Meerson,Zhang}.

Stochastic heat conduction equation offers a more accurate depiction of heat transfer processes affected by microscopic inhomogeneities in materials, thereby enriching the theoretical framework of heat conduction \cite{BCL,BG,BGHZ,CK,HNS}. In \cite{C}, R. Chiba reviewed the stochastic analysis of heat conduction and thermal stress in solids from six aspects: random surface/ambient temperature, random initial temperature, random material properties, random heat transfer coefficient, random internal heat generation/absorption, and random geometry. In \cite{H}, Y. Hu presented some recent advances in the stochastic heat equation subject to Gaussian noises with diverse variance structures, including the existence and uniqueness of the solution, the probabilistic representation of the integer-order moments and $p$-th moment boundedness ($p\geq2$) of the solution, as well as the exact asymptotic behavior of the solution in the sense of moments. And the stochastic heat equation with multiplicative noise is related to parabolic Anderson model, the Kardar--Parisi--Zhang (KPZ) equation or polymers (see \cite{H,KPZ,Q} and references therein).

The probability density function (PDF) of the solution process for an SPDE, along with the associated Fokker-Planck equation it satisfies, has garnered widespread attention due to their capability to reflect the statistical characteristics of the underlying stochastic process \cite{DaPrato,Agram,BV}. However, in contrast to the abundant research achievements in the field of stochastic differential equations (SDEs) \cite{Arnold,Bogachev,Henry,Meng,Pirner,Runfola}, the explicit PDF and Fokker--Planck equation for stochastic heat equation (SHE)  with an additive $Q$-Wiener noise under non-homogeneous boundary conditions have not yet been derived based on the literature currently at our disposal.

In this paper, we focus on the following SHE with an additive $Q$-Wiener noise and non-homogeneous boundary conditions:
\begin{eqnarray}\label{1.1}
\left\{
\begin{array}{ll}
\frac{\partial}{\partial t}U(t,x)=a^2\frac{\partial^2}{\partial x^2}{U(t,x)}+b_{}U(t,x)+f(t,x)+\sigma\frac{\partial}{\partial t}{W}(t,x),~&t>0,~x\in(0,1),\\ [0.2cm]
\frac{\partial}{\partial x}U(t,0)=h(t),~~U( t,1)=g(t),~&t\geq0,\\ [0.2cm]
U(0,x)=U_0(x),~&x\in(0,1),
\end{array}
\right.
\end{eqnarray}
where $a,b$ and $\sigma$ are constants, $f:[0,\infty)\times[0,1]\rightarrow \mathbb{R}$, $h,g:[0,\infty)\rightarrow \mathbb{R}$ are Borel measurable, and $W(t,x)$ is a $Q$-Wiener process. For more details on $Q$-Wiener process, we refer to Section 2.

 Furthermore, we consider the following SHE involving anomalous diffusion term with a multiplicative $Q$-Wiener noise and sinusoidal initial excitation:
\begin{align}\label{1.2}
\left\{
\begin{array}{ll}
\frac{\partial}{\partial t}{U}(t,x)=a^2\frac{\partial^2}{\partial x^2}U(t,x)-b(-\frac{\partial^2}{\partial x^2})^{\frac{\alpha}{2}}U(t,x)+cU(t,x)\\ [0.2cm]
\qquad\qquad\quad+\varepsilon U(t,x)\frac{\partial}{\partial t}{W}(t,x), &t>0,~x\in(0,1),\\ [0.2cm]
U( t,0)=U( t,1)=0,&t\geq0,\\ [0.2cm]
U( 0,x)=U_m(0)\sqrt{2}\sin(m\pi x),&x\in(0,1),
\end{array}
\right.
\end{align}
where $\alpha\in(0,2)$, $c$ is a constant, $\varepsilon$ is a positive real number, $U_m(0)>0$ is a random variable that is independent of the $Q$-Wiener process ${W}(t,x)$, and $m$ is a positive integer. Note that the presence of the non-local term in (\ref{1.2}) restricts our investigation to cases with homogeneous boundary conditions, while the emergence of multiplicative noise confines our consideration to scenarios where the initial condition only involves the $m$-th eigenfunction.

In addition, introduce the Cole--Hopf transformation $K(t,x)=2\theta\ln |U(t,x)|/\xi $, we investigate the following KPZ equation, which is related to the SHE (\ref{1.2}) with $a^2=\theta,b=c=0$ and $x\in(\underline{x},\overline{x})\subset(k/m,(k+1)/m)$, $k=0,1,2,\ldots,m-1$,
\begin{align}\label{1.3}
\left\{
\begin{array}{ll}
\frac{\partial}{\partial t}{K}(t,x)=\theta\frac{\partial^2}{\partial x^2}K(t,x)+\frac{\xi}{2}|\frac{\partial}{\partial x}K(t,x)|^2+\frac{2\theta\varepsilon}{\xi} \frac{\partial}{\partial t}{W}(t,x),&t>0,~x\in(\underline{x},\overline{x}),\\ [0.2cm]
K( t,\underline{x})=\frac{2\theta}{\xi}\ln |U( t,\underline{x})|,~~K( t,\overline{x})=\frac{2\theta}{\xi}\ln |U( t,\overline{x})|,&t\geq0,\\ [0.2cm]
K( 0,x)=\frac{2\theta}{\xi}\ln [U_m(0)\sqrt{2}|\sin(m\pi x)|],&x\in(\underline{x},\overline{x}).
\end{array}
\right.
\end{align}

This work consists of four main components:

(i) The first part centers on deriving the explicit PDF of the SHE (\ref{1.1}) with an additive $Q$-Wiener noise and non-homogeneous boundary conditions. The derivations employ the homogenization technique, the method of eigenfunction expansion and address the challenges posed by the space-dependence of the SPDE (\ref{1.1}) by using the fact that if each dimension of the random variables obeys a normal distribution and the sums of their means and variances are finite, then their sum (i.e., the solution of the stochastic system (\ref{1.1})) still obeys a normal distribution, with the mean and variance being the sums of the respective means and variances of the variables.

(ii) The second part is to derive the Fokker--Planck equation and the probability representation solution of the density function for the SHE (\ref{1.1}). It is worthy to mention that the probability representation solution of Fokker--Planck equation with time-dependence coefficients can be constructed by using the time inverse method and Feynman--Kac formula.

(iii) The third part focuses on the case of multiplication noise. It should be pointed out that the logarithm of the random variable $ U_m( 0)$ has a distribution possessing finite first and second moments, which includes but is not limited to the normal, uniform and exponential distributions. Additionally, if the random variable $ U_m( 0)$ is a constant, $\log U_m( 0)$ obeys a Dirac delta distribution. The appearance of multiplicative noise restricts our consideration to results under the initial condition corresponding to the $m$-th eigenfunction. Because in this situation, the random variables in each dimension obey a log-normal distribution, which prevents us from obtaining the distribution of their sum.

(iv) The fourth part is devoted to studying the differences in the PDFs of the solution processes for SDEs and SPDEs. Compared to the PDF of an SDE, the evolution of the PDF for an SPDE is more complex, i.e., the evolution of the PDF for the solution $U(t,x)$ of the SHE (\ref{1.2}) depends critically on the choice of the spatial location $x$. Specifically, when $x\in\{0,1/m,\ldots,(m-1)/m,1\}$, the random variable $U(t,x)$ obeys a Dirac delta distribution for every $t$, and when $x\in(0,1)\backslash\{0,1/m,\ldots,(m-1)/m,1\}$, the random variable $U(t,x)$ obeys a log-normal distribution with spatiotemporally dependent mean and purely time-dependent variance. Moreover, as $x$ approaches $k/m$, $k=1,2,\ldots,m-1$, the distribution of $U(t,x)$ transitions from the log-normal to the Dirac delta distribution that is independent of time.

The paper is organized as follows. In Section 2, we derive the explicit PDF, the corresponding Fokker--Planck equation, and probability representation solution for the SHE (\ref{1.1}). In Section 3, we establish analogous results for the SHE (\ref{1.2}) and analyze the PDF at different spatial locations. In Section 4, we investigate the KPZ equation (\ref{1.3}) related to the SHE (\ref{1.2}). Additionally, we provide some numerical examples to illustrate our theoretical results.

\section{Stochastic heat equation with an additive $Q$-Wiener noise under non-homogeneous boundary conditions}\label{sec2}

In this section, we discuss the PDF and the Fokker--Planck equation of (\ref{1.1}) with an additive $Q$-Wiener process $W(t,x)=\sum_{n=1}^\infty q_nW_n(t)e_n(x)$, where $q_n\geq0$ and $\{W_n(t)\}_{n\geq1}$ are standard scalar independent Brownian motions. For $Q$-Wiener process $W$, the trace of the covariance operator $Q$ satisfies $\mathrm{Tr}(Q)=\sum_{n=1}^\infty q_n^2<\infty$. For more details on $Q$-Wiener process, please refer to \cite[Section 3.5]{DW}.

We first make the following assumption.

\smallskip
($\mathbf{A0}$) The functions $g$ and $h$ in (\ref{1.1}) belong to  $C^1([0,\infty))$, and ${f}$ belongs  to $ C^1([0,\infty);L^1([0,1]))$.

\smallskip
Let $Y(t,x)=h(t)[x-1]+g(t)$ and $V(t,x)=U(t,x)-Y(t,x)$. If $U(t,x)$ satisfies (\ref{1.1}), then $V(t,x)$ satisfies the following SHE with homogeneous boundary conditions
\begin{align}\label{2.1}
  \left\{
\begin{array}{ll}
\frac{\partial}{\partial t}{V}(t,x)=a^2\frac{\partial^2}{\partial x^2}V(t,x)+bV(t,x)+\tilde{f}(t,x)+\sigma\frac{\partial}{\partial t}{W}(t,x),&x\in(0,1),\\ [0.2cm]
\frac{\partial}{\partial x}V( t,0)=V( t,1)=0,&t\geq0,\\ [0.2cm]
V( 0,x)=U_0(x)-Y(0,x)=:V_0(x),&x\in(0,1),
\end{array}
\right.
\end{align}
where $\tilde{f}(t,x)=f(t,x)-\frac{\partial}{\partial t}Y(t,x)+bY(t,x)$. Conversely, suppose $V(t,x)$ satisfies (\ref{2.1}), then $U(t,x)=V(t,x)+Y(t,x)$ satisfies (\ref{1.1}) as well. Therefore, SHEs (\ref{1.1}) and (\ref{2.1}) are equivalent.

For every $n$, set $e_n(x)=\sqrt{2}\cos(\beta_n x)$ with $\beta_n=\left(n-1/2\right)\pi$ and denote
\begin{align}\label{2.2}
\lambda_n=b-(a\beta_n)^2,~\tilde{f}_n(t)=\langle \tilde{f}(t,\cdot),e_n(\cdot)\rangle~ \mathrm{and} ~Y_n(t,x
)=\langle Y(t,\cdot),e_n(\cdot)\rangle e_n(x).
\end{align}

\begin{lemma}\label{le.2.1}
Let $V$ be the solution of $(\ref{2.1})$ and $V_n(t)=\langle V( t,\cdot),e_n(\cdot)\rangle$. Then
\begin{align}\label{2.3}
 {V}_n(t)=V_n(0)\exp(\lambda_n t)+\int_0^t\exp\{\lambda_n (t-s)\}\tilde{f}_n(s)\mathrm{d}s+\sigma q_n\int_0^t\exp\{\lambda_n (t-s)\}\mathrm{d}{W}_n(s)
\end{align}
and $V(t,x)=\sum_{n=1}^\infty V_n(t)e_n(x)$.
\end{lemma}
\begin{proof} Using the eigenfunction expansion with eigenfunctions $e_n(x)=\sqrt{2}\cos(\beta_n x)$, $n=1,2,\ldots$, we have $V(t,x)=\sum_{n=1}^\infty V_n(t)e_n(x)$. If one multiplies both sides of the first identity in (\ref{2.1}) by $e_n$ and subsequently integrates the resulting equation over the spatial domain $[0,1]$, it leads to
\begin{align*}
 \mathrm{d}{V}_n(t)=\big[\lambda_nV_n(t)+\tilde{f}_n(t)\big]\mathrm{d}t+\sigma q_n\mathrm{d}{W}_n(t),~~n=1,2,\ldots,
\end{align*}
where  $\lambda_n$, $\tilde{f}_n(t)$ are given in (\ref{2.2}) and the initial value $V_n(0)=\langle V(0,\cdot),e_n(\cdot)\rangle$. Clearly, its solution is (\ref{2.3}).
\end{proof}

Let $V_n(t,x)=V_n(t)e_n(x)$ and $\tilde{f}_n(t,x)=\tilde{f}_n(t) e_n(x)$. By (\ref{2.3}), it follows that
\begin{align}\label{2.4}
 \nonumber{V}_n(t,x)=&V_n(0,x)\exp(\lambda_n t)+\int_0^t\exp\{\lambda_n (t-s)\}\tilde{f}_n(s,x)\mathrm{d}s\\
 &+\sigma q_ne_n(x)\int_0^t\exp\{\lambda_n (t-s)\}\mathrm{d}{W}_n(s).
\end{align}
Denote $U_n(t,x)=\langle U( t,\cdot),e_n(\cdot)\rangle e_n(x)$. In view of Lemma \ref{le.2.1} and (\ref{2.4}), we have
\begin{align}\label{2.5}
U_n(t,x)=V_n(t,x)+Y_n(t,x) \quad
{\rm and} \quad
 U( t,x)=\sum_{n=1}^{\infty}U_n(t,x).
\end{align}
Without loss of generality, we assume that

\smallskip
($\mathbf{A1}$) The initial value of the stochastic process $\{{U}_n(t):=\langle U(t,\cdot),e_n(\cdot)\rangle\}_{t\geq 0}$ obeys the normal distribution $N({\mu}_n(0),\nu_n(0))$, where the constants ${\mu}_n(0)\in\mathbb{R}$ and $\nu_n(0)\in(0,\infty)$. Moreover,
\begin{align}\label{2.6}
\sum_{n=1}^\infty|\mu_n(0)|<\infty \quad {\rm and} \quad \sum_{n=1}^\infty|\nu_n(0)|<\infty.
\end{align}

It follows from (\ref{2.4}), (\ref{2.5}) and the assumption ($\mathbf{A1}$) that the mean
\begin{align}\label{2.7}
\nonumber\mathbb{E}{U}_n(t,x)&=\mathbb{E}[{V}_n(0,x)]\exp(\lambda_n t)+\int_0^t\exp\{\lambda_n (t-s)\}\tilde{f}_n(s,x)\mathrm{d}s+Y_n(t,x)\\
\nonumber&=[{\mu}_n(0)e_n(x)-{Y}_n(0,x)]\exp(\lambda_n t)+\int_0^t\exp\{\lambda_n (t-s)\}\tilde{f}_n(s,x)\mathrm{d}s\\
&~~~+Y_n(t,x)=:{\mu}_n(t,x),~~n=1,2,\ldots
\end{align}
and the variance
\begin{align}\label{2.8}
 \nonumber Var({U}_n(t,x))&=Var(U_n(0,x))\exp(2\lambda_n t)-\frac{[\sigma q_ne_n(x)]^2}{2\lambda_n}\left[1-\exp(2\lambda_n t)\right]\\
 &=\nu_n(0)e_n^2(x)\exp(2\lambda_n t)+\frac{[\sigma q_ne_n(x)]^2}{2\lambda_n}\left[\exp(2\lambda_n t)-1\right]=:\nu_n(t,x),
\end{align}
if $\lambda_n\neq0$. Note that if there exists a positive integer $l$ such that $\lambda_l=0$, then $\nu_l(t,x)=Var({U}_l(t,x))=\nu_l(0)e_l^2(x)+[\sigma q_le_l(x)]^2t$.

In order to derive the PDF and Fokker--Planck equation of the SHE (\ref{1.1}) with the additive $Q$-Wiener noise and non-homogeneous boundary conditions, we now show that $\sum_{n=1}^{\infty}{\mu}_n(t,x)<\infty$ and $\sum_{n=1}^{\infty}\nu_n(t,x)<\infty$ under the assumptions ($\mathbf{A0}$)--($\mathbf{A1}$) in the following lemma.

\begin{lemma}\label{le.2.2}
Suppose that $(\mathbf{A0})$--$(\mathbf{A1})$ hold.  Then for every $(t,x)\in (0,\infty)\times [0,1]$, $\sum_{n=1}^{\infty}\mu_n(t,x)<\infty$ and $\sum_{n=1}^{\infty}\nu_n(t,x)<\infty$.
\end{lemma}
\begin{proof}
For every $(t,x)\in (0,\infty)\times [0,1]$, it follows from (\ref{2.7}), that
\begin{eqnarray}\label{2.9}
\sum_{n=1}^\infty\mu_n(t,x) &\leq& |Y(t,x)|+\sum_{n=1}^\infty|{\mu}_n(0)|+ \sum_{n=1}^\infty\Big\{\Big|\int_0^1{Y}(0,x)e_n(x)\mathrm{d}x\Big|\exp\{\lambda_n t\} \nonumber \\
&&+\int_0^t\Big|\int_0^1\tilde{f}(s,x)e_n(x)\mathrm{d}x\Big|\exp\{\lambda_n (t-s)\}\mathrm{d}s\Big\}
\nonumber \\
&\leq&\max_{x\in [0,1]}|Y(t,x)|+\sum_{n=1}^\infty|{\mu}_n(0)|+ \sum_{n=1}^\infty\sqrt{2}\Big\{\max_{x\in [0,1]}|Y(0,x)|\exp\{\lambda_n t\} \nonumber \\
&&+\max_{(s,x)\in [0,t]\times [0,1]}|\tilde{f}(s,x)| \int_0^t\exp\{\lambda_ns\}\mathrm{d}s\Big\}<\infty,
\end{eqnarray}
where in the second inequality we have used the fact that  $\tilde{f}(t,x)$ and $Y(t,x)$ are continuous in $[0,\infty)\times [0,1]$,  $\lambda_n=b-a^2(n-1/2)^2\pi^2$ and (\ref{2.6}).

For $\nu_n(t,x)$, one estimates from (\ref{2.8}) that
\begin{align}\label{2.10}
\sum_{n=1}^\infty{\nu}_n(t,x)&\leq 2\sum_{n=1}^\infty\nu_n(0)+\sum_{n=1}^\infty\frac{\sigma^2 q_n^2}{\lambda_n}\left[\exp(2\lambda_n t)-1\right].
\end{align}
Observe that $\lambda_n=b-a^2(n-1/2)^2\pi^2$, then there exists $l\in \mathbb{N}$ such that $\lambda_l<-1$. This, together with (\ref{2.10}), yields to
\begin{align}\label{2.11}
\sum_{n=1}^\infty{\nu}_n(t,x)&\leq 2\sum_{n=1}^\infty\nu_n(0)+\sum_{n=1}^{l-1}\frac{\sigma^2 q_n^2|\exp(2\lambda_n t)-1|}{|\lambda_n|}+\sum_{n=l}^\infty\sigma^2 q_n^2
<\infty.
\end{align}
By (\ref{2.9}) and (\ref{2.11}), we complete the proof.
\end{proof}

Let $X_1,X_2,\ldots$ be independent random variables with distributions $F_1,F_2,\ldots$, respectively. It is assumed that $\mathbb{E}X_k=0$ and that $\eta_k^2=\mathbb{E}(X_k^2)$ exists. The following lemma is given for the proof of Theorem \ref{th.2.4} below (see \cite[Section VIII.5]{Feller}).
\begin{lemma}\label{le.2.3}
If $\eta^2=\sum_{k=1}^\infty\eta_k^2<\infty$, the distributions $G_n$ of the partial sums $X_1+\ldots+X_n$ tend to a probability distribution $G$ with zero expectation and variance $\eta^2$.
\end{lemma}

\begin{theorem}\label{th.2.4}
 Let ${\mu}(t,x)=\sum_{n=1}^{\infty}{\mu}_n(t,x)$ and $\nu(t,x)=\sum_{n=1}^{\infty}\nu_n(t,x)$. Under the assumptions of Lemma \ref{le.2.2}, the PDF of the SHE $(\ref{1.1})$ is
\begin{align}\label{2.12}
p(u,t,x)=\frac{1}{\sqrt{2\pi\nu(t,x)}}\exp\left\{-\frac{(u-{\mu}(t,x))^2}{2\nu(t,x)}\right\},~~~~~~~~~t\geq0,
\end{align}
and it satisfies the following Fokker--Planck equation
\begin{align}\label{2.13}
  \left\{
\begin{array}{ll}
\frac{\partial}{\partial t}p(u,t,x)={\mathcal{M}}(t,x)\frac{\partial}{\partial u}p(u,t,x)+\frac{1}{2}{\mathcal{G}}(t,x)\frac{\partial^2}{\partial u^2}p(u,t,x),\\ [0.2cm]
p(u,0,x)=\frac{1}{\sqrt{2\pi\nu(0,x)}}\exp\left\{-\frac{(u-{\mu}(0,x))^2}{2\nu(0,x)}\right\},
\end{array}
\right.
\end{align}
where ${\mathcal{M}}(t,x)=-\frac{\partial{\mu}(t,x)}{\partial t}$ and ${\mathcal{G}}(t,x)=\frac{\partial\nu(t,x)}{\partial t}$. Furthermore, the transition probability density $p(u,t,x|w,s,x)$ of the process $U$ satisfies the following Chapman--Kolmogorov equation
\begin{align}\label{2.14}
 p(u,t,x|w,s,x)=\int_{\mathbb{R}}p(u,t,x |v,r,x)p(v,r,x |w,s,x)dv ,~~0<s<r<t.
\end{align}
\end{theorem}
\begin{proof}
By (\ref{2.7})-(\ref{2.8}), $U_n(t,x)$ obeys the normal distribution $N({\mu}_n(t,x),\nu_n(t,x))$, $t>0$. Define $\mathcal{Y}_n(t,x)=U_n(t,x)-{\mu}_n(t,x)$. Then $\mathcal{Y}_n(t,x)$ is distributed as $N(0,\nu_n(t,x)),~t>0$. With the aid of Lemmas \ref{le.2.2} and \ref{le.2.3}, $\sum_{k=1}^n\mathcal{Y}_k(t,x)$ tends to a probability distribution with zero expectation and variance $\nu(t,x)$. It follows from (\ref{2.5}) and Lemma \ref{le.2.2} that the random variable $U(t,x)$ is distributed with expectation ${\mu}(t,x)$ and variance $\nu(t,x)$.

Let $\varphi_{U_k}(\theta,x) = \exp({i{\mu}_k(t,x)\theta - \frac{1}{2}\nu_k(t,x)\theta^2})$ be the characteristic function of $U_k(t,x)$. Then, the characteristic function of $ S_n(t,x)=\sum_{k = 1}^n U_k(t,x)$ is
    \begin{equation*}
        \varphi_{S_n}(\theta,x) = \prod_{k = 1}^n \varphi_{U_k}(\theta,x) = \exp\Big({i\theta\sum_{k = 1}^n {\mu}_k(t,x) - \frac{1}{2}\theta^2\sum_{k = 1}^n \nu_k(t,x)}\Big).
    \end{equation*}
Thanks to Lemma \ref{le.2.2}, we have
    \begin{equation*}
        \lim_{n\rightarrow\infty} \varphi_{S_n}(\theta,x) = \exp\big({i\theta{\mu}(t,x) - \frac{1}{2}\theta^2{\nu}(t,x)}\big),
    \end{equation*}
which is the characteristic function of the normal distribution $N\left({\mu}(t,x), {\nu}(t,x)\right)$. Therefore the PDF of the SHE (\ref{1.1}) is (\ref{2.12}). Moreover, (\ref{2.13}) holds by simple calculations.

On the other hand, it follows from (\ref{2.4}) and (\ref{2.5}) that the stochastic process $\{U_n(t,x)\}_{t\geq 0}$ is a Markov process, i.e, it satisfies $\mathbf{P}\{U_n(t,x)\in A|\mathcal{F}_{n,s}\}=\mathbf{P}\{U_n(t,x)\in A|U_n(s,x)\}$, where $A$ is a Borel measurable set in $\mathbb{R}$ and the $\sigma$-algebra $\mathcal{F}_{n,s}$ is generated by $\{U_n(r,x)|0\leq r\leq s\}$. In view of $U(t,x)=\sum_{n=1}^\infty U_n(t,x)$ and the independence of $W_n(t)$,  $U(t,x)$ is a Markov process. Therefore, the Chapman-Kolmogorov equation (\ref{2.14}) holds.
\end{proof}

\begin{sidewaystable}
    \centering 
    \caption{Non-homogeneous boundary conditions and the corresponding $Y(t,x)$ }
\renewcommand{\arraystretch}{2.5}
\begin{tabular}{|c|c|c|c|}
  \hline
    Case& Boundary Condition at $x=0$  & Boundary Condition at $x=1$ & $Y(t,x)$ \\ \hline
   1 & $U(t,0)=g(t)$ & $U(t,1)=h(t)$ & $[h(t)-g(t)]x+g(t)$ \\ \hline
   2 & $U(t,0)=g(t)$ & $\frac{\partial}{\partial x}U(t,1)=h(t)$ & $h(t)x+g(t)$ \\ \hline
   3 & $\frac{\partial}{\partial x}U(t,0)=g(t)$ & $\frac{\partial}{\partial x}U(t,1)=h(t)$ & $\frac{h(t)-g(t)}{2}x^2+g(t)x$\\ \hline
   4 & $U(t,0)=g(t)$ & $(\frac{\partial}{\partial x}U+\gamma U)(t,1)=h(t)$ & $\frac{h(t)-\gamma g(t)}{1+\gamma}x+g(t)$ \\ \hline
   5 &  $\frac{\partial}{\partial x}U(t,0)=g(t)$ & $(\frac{\partial}{\partial x}U+\gamma U)(t,1)=h(t)$ & $g(t)x+\frac{h(t)-(1+\gamma) g(t)}{\gamma}$ \\ \hline
   6 & $(\frac{\partial}{\partial x}U-\gamma U)(t,0)=g(t)$ & $U(t,1)=h(t)$ & $\frac{g(t)+\gamma h(t)}{1+\gamma}x+\frac{h(t)-g(t) }{1+\gamma}$ \\ \hline
   7 & $(\frac{\partial}{\partial x}U-\gamma U)(t,0)=g(t)$ & $\frac{\partial}{\partial x}U(t,1)=h(t)$ & $h(t)x+\frac{h(t)-g(t)}{\gamma}$ \\ \hline
   8 & $(\frac{\partial}{\partial x}U-\gamma_1 U)(t,0)=g(t)$ & $(\frac{\partial}{\partial x}U+\gamma_2 U)U(t,1)=h(t)$ & $\frac{\gamma_1h(t)+\gamma_2g(t)}{\gamma_1+\gamma_2+\gamma_1\gamma_2}x+\frac{h(t)-(1+\gamma_2)g(t)}{\gamma_1+\gamma_2+\gamma_1\gamma_2}$ \\
  \hline
\end{tabular}
\end{sidewaystable}
\begin{remark}
\rm
In addition to the non-homogeneous boundary conditions in (\ref{1.1}), the functions $Y(t,x)$ for other non-homogeneous boundary conditions are given in the following table. For the equation $\frac{\partial^2}{\partial x^2}V(x)+\beta V(x)=0,~0<x<1$, the eigenfunction of $\frac{\partial^2}{\partial x^2}$ with homogeneous boundary conditions in (\ref{2.2}) is $e_n(x)=\sqrt{2}\cos(\beta_n x)$, $n=1,2,\ldots$, where $\beta_n=\left(n-1/ {2}\right)\pi$. Moreover, we can find the corresponding eigenvectors for other homogeneous boundary conditions, e.g., $V(0)=V(1)=0$, $V(0)=\frac{\partial}{\partial x}V(1)=0$, $\frac{\partial}{\partial x}V(0)=\frac{\partial}{\partial x}V(1)=0$, $V(0)=\frac{\partial}{\partial x}V(1)+\gamma V(1)=0$, $\frac{\partial}{\partial x}V(0)=\frac{\partial}{\partial x}V(1)+\gamma V(1)=0$, $\frac{\partial}{\partial x}V(0)-\gamma V(0)=V(1)=0$, $\frac{\partial}{\partial x}V(0)-\gamma V(0)=\frac{\partial}{\partial x}V(1)=0$ and $\frac{\partial}{\partial x}V(0)-\gamma_1 V(0)=\frac{\partial}{\partial x}V(1)+\gamma_2 V(1)=0$. Therefore, for the SHE under non-homogeneous boundary conditions as presented in Table $1$, we can still use the method of eigenfunction expansion to obtain an explicit expression for its PDF, as well as the corresponding Fokker--Planck equation.
\end{remark}
\begin{remark}\label{re.2.6}
\rm
In addition to the normal distribution mentioned in assumption $(\mathbf{A1})$, the initial value $U_n( 0)$ can obey other distributions with finite first and second moments, including but not limited to the uniform distribution and the exponential distribution. Additionally, if $U_n( 0)$ is a constant, it follows a Dirac delta distribution centered at that constant value.
\end{remark}
To give the probabilistic representation for the solution of (\ref{2.13}) with time-dependent coefficients, for $u\in \mathbb{R}$, $x\in (0,1)$ and $\tilde{t}\in[0,T]$, define $\tilde{p}(u,\tilde{t},x):=p(u,T-\tilde{t},x),$ $\tilde{\mathcal{M}}(\tilde{t},x):=\mathcal{M}( T-\tilde{t},x)$, and $\tilde{\mathcal{G}}(\tilde{t},x):=\mathcal{G}(T-\tilde{t},x)$. Then $\tilde{p}(u,\tilde{t},x)$ satisfies the following backward Fokker-Planck equation
\begin{align}\label{2.15}
\frac{\partial \tilde{p}(u,\tilde{t},x)}{\partial \tilde{t}}\!+\frac{1}{2}\tilde{\mathcal{G}}(\tilde{t},x)\frac{\partial^2 \tilde{p}(u,\tilde{t},x)}{\partial{u}^2}+\tilde{\mathcal{M}}(\tilde{t},x)\frac{\partial \tilde{p}(u,\tilde{t},x)}{\partial{u}}=0,
\end{align}
with the terminal value $\tilde{p}(u,T, x)=p(u,0,x).$ In view of $\tilde{p}(u,\tilde{t},x)=p(u,T-\tilde{t},x)$ and (\ref{2.12}), we have
\begin{align}\label{2.16}
\tilde{p}(u,\tilde{t},x)=\frac{1}{\sqrt{2\pi\nu(T-\tilde{t},x)}}\exp\left\{-\frac{(u-{\mu}(T-\tilde{t},x))^2}{2\nu(T-\tilde{t},x)}\right\}.
\end{align}

From the above equality, the solution of the equation (\ref{2.15}) exists. Moreover, by the similar arguments as Theorem 3.8 in \cite{Meng}, the uniqueness of the solution for the backward Fokker--Planck equation (\ref{2.15}) follows. Let $\tilde{t}=T-t$. Then $p(u,t,x)=\tilde{p}(u,T-t,x)$ and the solution ${p}(u,{t},x)$ of the Fokker--Planck equation (\ref{2.13}) is unique. The following theorem provides an alternative form of the solution $p(u,t,x)$ through a probability expression.
\begin{theorem}\label{th.2.7}
Under the assumption of Theorem \ref{th.2.4}. Suppose that $$\sum\limits_{n=1}^\infty[2\lambda_n\nu_n(0)+(\sigma q_n)^2]e_n^2(x)\exp(2\lambda_n s)>0$$ for $(s,x)\in[0,T)\times(0,1)$ and $\hat{U}(s,x)$ is a solution of the following SDE
\begin{align}\label{2.17}
 d{\hat{U}( s,x)}=\tilde{\mathcal{M}}( s,x){ds}+\sqrt{\tilde{\mathcal{G}}(s,x)}dB(s),~T-t\leq s\leq T,~~x\in(0,1)
  \end{align}
with the initial condition $\hat{U}(T-t,x)=u$. Then the density function $p$ of the SHE  $(\ref{1.1})$ admits the probabilistic representation
\begin{align}\label{2.18}
p(u, t,x)=\mathbb{E}\big[p(\hat{U}(T,x),0,x)|{\hat{U}( T-{t},x)=u}\big],~~( u ,t,x)\in\mathbb{R}\times[0,T]\times[0,1].
\end{align}
\end{theorem}
\begin{proof}
Applying It\^{o}'s formula to the function $\bar{Y}(s)=\tilde{p}(\hat{U}(s,x), s,x)$, it follows that
\begin{align*}
\begin{array}{clc}
\mathrm{d}\bar{Y}(s)&=\left[\frac{\partial\tilde{p}(\hat{U}(s,x), s,x)}{\partial s}+\tilde{\mathcal{M}}( s,x)\frac{\partial\tilde{p}(\hat{U}(s,x), s,x)}{\partial u}+\frac{1}{2}\tilde{\mathcal{G}}(s,x)\frac{\partial^2\tilde{p}(\hat{U}(s,x), s,x)}{\partial u^2}\right]\mathrm{d}s\\
&\quad +\frac{\partial\tilde{p}(\hat{U}(s,x), s,x)}{\partial u}\sqrt{\tilde{\mathcal{G}}(s,x)}\mathrm{d}B(s)
\\ &=\frac{\partial\tilde{p}(\hat{U}(s,x), s,x)}{\partial u}\sqrt{\tilde{\mathcal{G}}(s,x)}\mathrm{d}B(s),
\end{array}
\end{align*}
where in the second identity we have used  (\ref{2.15}).

Thanks to Dynkin's formula,
\begin{align*}
\begin{array}{clc}
&~~~\mathbb{E}[\bar{Y}(T\wedge\tau_R)|\hat{U}(\tilde{t},x)=u]\\
&=\mathbb{E}[\bar{Y}(\tilde{t})|\hat{U}(\tilde{t},x)=u]+\mathbb{E}[\int_{\tilde{t}}^{T\wedge\tau_R}\frac{\partial\tilde{p}(\hat{U}(s,x), s,x)}{\partial u}\sqrt{\tilde{\mathcal{G}}(s,x)}\mathrm{d}B(s)|\hat{U}(\tilde{t},x)=u]~~~~~~~~~~~~~~~~~~~~~~~~~~\\
&=\mathbb{E}[\bar{Y}(\tilde{t})|\hat{U}(\tilde{t},x)=u]=\tilde{p}(u, \tilde{t},x),
\end{array}
\end{align*}
where $\tau_R=\inf\{s>0:|\hat{U}(s,x)|>R\}$ for all $R>0$. From $\nu(s,x)>0$ for $(s,x)\in[0,T)\times(0,1)$ and (\ref{2.16}), we obtain that $\tilde{p}$ is bounded, and
\begin{align*}
\begin{array}{clc}
\tilde{p}(u, \tilde{t},x)=\lim\limits_{R\rightarrow\infty}\mathbb{E}
[\bar{Y}(T\wedge\tau_R)|\hat{U}(\tilde{t},x)=u]=\mathbb{E}\big[ \tilde{p}(\hat{U}(T,x),T,x)|{\hat{U}( \tilde{t},x)=u}\big].
\end{array}
\end{align*}
Therefore, the assertion of this theorem follows immediately by letting $\tilde{p}(u,\tilde{t},x)=p(u,T-\tilde{t},x)$ and $\tilde{t}=T-t$.
\end{proof}
\begin{remark}
\rm
Note that the condition $\frac{\partial}{\partial s}\nu(s,x)=\sum_{n=1}^\infty[2\lambda_n\nu_n(0)+(\sigma q_n)^2]e_n^2(x)\exp(2\lambda_n s)>0$ ensures that $\tilde{\mathcal{G}}(s,x)>0$ for $(s,x)\in[0,T)\times(0,1)$.
\end{remark}

\begin{example}
\rm
Let $a=1,b=1,\alpha=1/2,f(t,x)=\cos(t)e_1(x),\sigma=1,q_n=1/n,h(t)=\cos(t),g(t)=\sin(t)$ and $U_0(x)=U_1(0)e_1(x)$ where $e_n(x)=\sqrt{2}\cos(\beta_nx)$, $\beta_n=(n-1/2)\pi$, $n=1,2,\ldots,$ and the random variable $U_1(0)$ obeys the normal distribution $N(0,1/16)$.  We deduce from (\ref{2.7}) and (\ref{2.8}) that
\begin{align}\label{2.19}
\nonumber{\mu}(t,x)&=\sum_{n=1}^{\infty}{\mu}_n(t,x)=\cos(t)[x-1]+\sin(t)+\sum_{n=1}^{\infty}\frac{\sqrt{2}}{\beta_n}\left[\frac{1}{\beta_n}+(-1)^{n}\right]\exp(\lambda_n t)\\
&~~~+\frac{\sin(t)-\cos(t)+\lambda_n\exp(\lambda_nt)}{1+\lambda_n^2}\left\{I_{\{n=1\}}+\frac{\sqrt{2}}{\beta_n}\left[(-1)^{n+1}-\frac{1}{\beta_n}\right]\right\}
\end{align}
and
\begin{align}\label{2.20}
\nu(t,x)=\sum_{n=1}^{\infty}\nu_n(t,x)=\frac{1}{16}e_1^2(x)\exp(2\lambda_1 t)+\sum_{n=1}^{\infty}\frac{e_n^2(x)}{2\lambda_nn^2}\left[\exp(2\lambda_n t)-1\right],
\end{align}
where $\lambda_n=1-\beta_n^2$. Thanks to Theorem \ref{th.2.4}, the PDF of the SHE $(\ref{1.1})$ is
\begin{align}\label{2.21}
p(u,t,x)=\frac{1}{\sqrt{2\pi\nu(t,x)}}\exp\left\{-\frac{(u-{\mu}(t,x))^2}{2\nu(t,x)}\right\}, \quad t>0.
\end{align}
By Theorem \ref{th.2.7}, $p(u, t,x)$ also has the following probabilistic representation
\begin{align}\label{2.22}
p(u, t,x)=\mathbb{E}\big[ p(\hat{U}(T,x),0,x)|{\hat{U}(T-{t},x)=u}\big], \quad (u ,t,x)\in\mathbb{R}\times[0,T]\times[0,1],
\end{align}
where $p(u, 0,x)$ is given in (\ref{2.13}) and $\hat{U}\left( s,x\right)$ is an It\^{o} process satisfying
\begin{align*}
 d{\hat{U}( s,x)}=\tilde{\mathcal{M}}( s,x){ds}+\sqrt{\tilde{\mathcal{G}}(s,x)}dB(s), \quad  T-t\leq s\leq T,~~x\in(0,1)
  \end{align*}
with the initial value $\hat{U}( T-{t},x)=u$. Here
\begin{align*}
\tilde{\mathcal{M}}(s,x)&=\frac{\partial{\mu}( T-s,x)}{{\partial s}}=\sin(T-s)[x-1]-\cos(T-s)\\
&~~~-\sum_{n=1}^{\infty}\frac{\cos(T\!-\!s)\!+\!\sin(T\!-\!s)\!+\!\lambda_n^2\exp\{\lambda_n (T\!-\!s)\}}{1+\lambda_n^2}\left\{I_{\{n=1\}}\!+\!\frac{\sqrt{2}}{\beta_n}\left[(-1)^{n+1}\!-\!\frac{1}{\beta_n}\right]\right\}\\
&~~~+\sum_{n=1}^{\infty}\frac{\sqrt{2}\lambda_n}{\beta_n}\left[(-1)^{n+1}-\frac{1}{\beta_n}\right]\exp\{\lambda_n (T-s)\},\\
\tilde{\mathcal{G}}(s,x)&=-\frac{\partial\nu(T-s,x)}{\partial s}=\Big(\frac{36-\pi^2}{32}\Big)e_1^2(x)\exp\{2\lambda_1 (T-s)\}+\sum_{n=2}^{\infty}\frac{e_n^2(x)}{n^2}\exp\{2\lambda_n (T-s)\}.
\end{align*}

By choosing $10^4$ samples of $\hat{U}( t,x)$ and $n=1,2,\ldots,10$, Figure \ref{1} plots the PDFs $p(u, 1,x)$ of (\ref{1.1}) with different $x$, and diverse expressions (\ref{2.21}) and (\ref{2.22}). It shows that the evolution of the probability density for (\ref{1.1}) is related to the spatial position $x$.
\end{example}
\begin{figure}
\centering%
\includegraphics[height=2.5in,width=5in]{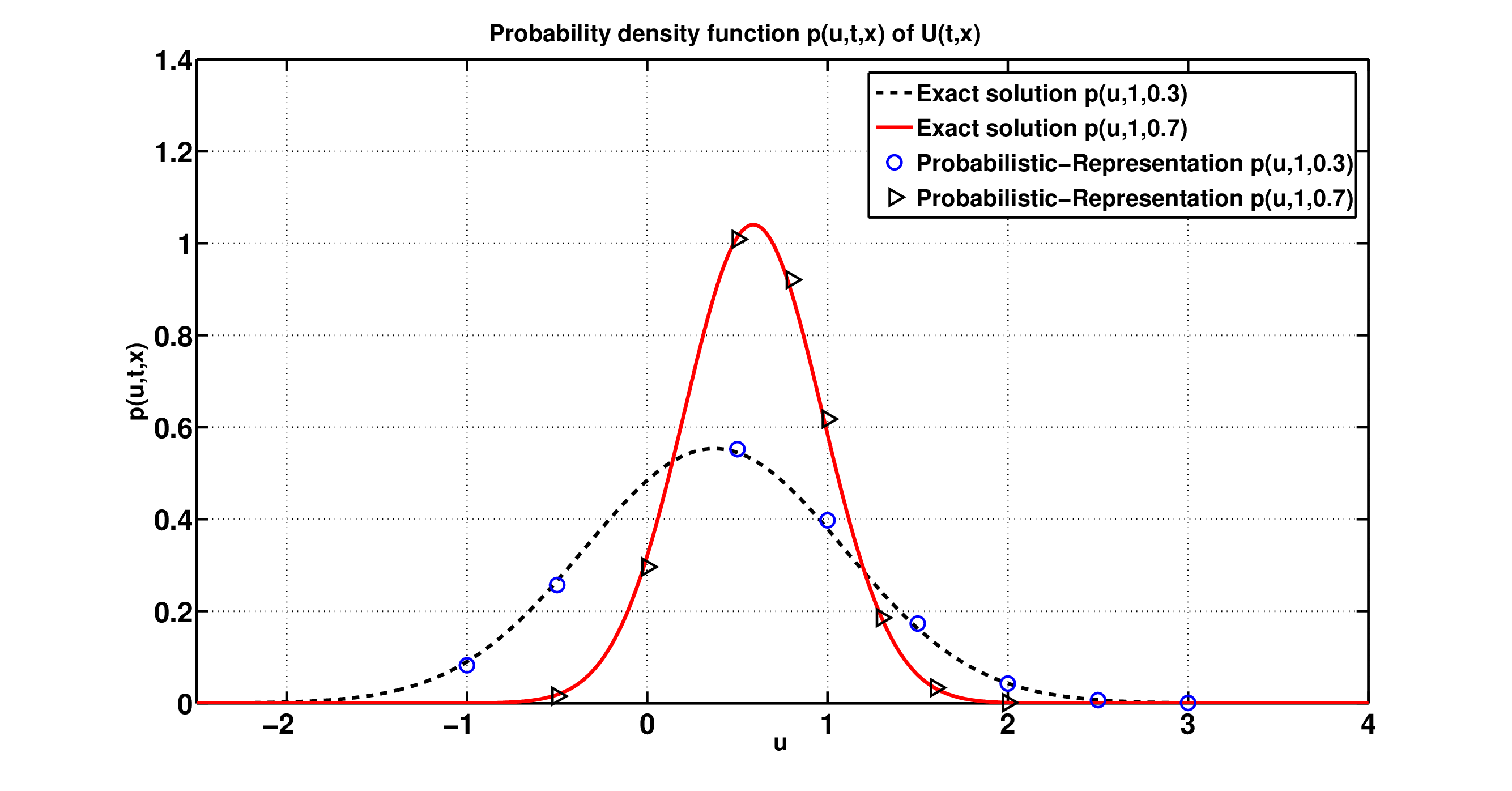}

\caption{Simulation of the PDF $p(u, 1,x)$ of (\ref{1.1}) with the step size $\triangle t=10^{-2}$, $T=2$, and different $u$ and $x$.}
\label{1}
\end{figure}

\section{Stochastic heat equation with a multiplicative $Q$-Wiener noise under homogeneous boundary conditions}\label{sec3}

In order to derive the PDF of the system (\ref{1.2}), denote
\begin{align}\label{3.1}
 U(t,x)=\sum_{n=1}^{\infty}U_n(t)\tilde{e}_n(x)=:\sum_{n=1}^{\infty}U_n(t,x),
\end{align}
where $\tilde{e}_n(x)=\sqrt{2}\sin(n\pi x)$, $n=1,2,\ldots.$ Let $m\in \mathbb{N}$ be given in (\ref{1.2}), we set
\begin{eqnarray*}\label{3.2}
\left\{\begin{array}{ll}
\mathbf{{\bar{D}}}=\{(u,x)|u\in\mathbb{R},x\in[0,1]\},
\quad \mathbf{\Gamma}=\left\{(u,x)|u\in\mathbb{R},x\in\{0,\frac{1}{m},\ldots,\frac{m-1}{m},1\}\right\} \\ [0.2cm]
 \mathbf{{D}}=
 \mathop{\cup}\limits_{k=0}^{\lfloor(m-1)/2\rfloor}
[\mathbf{{D}}_{1,m}^k\cup\mathbf{{D}}_{2,m}^k], \quad \mathbf{{D}}_{1,m}^k=\{(u,x)|u\in
        (0,\infty),$ $x\in\left(\frac{2k}{m}, \frac{2k+1}{m}\right)\}, \\ [0.2cm] \mathbf{{D}}_{2,m}^k=\{(u,x)|u\in
        (-\infty, 0),~x\in\left(\frac{2k+1}{m}, \frac{2k+2}{m}\right) \cap (0, 1)\},
\end{array}\right.
\end{eqnarray*}
where the floor function $\lfloor x\rfloor$, maps a real number $x$ to the greatest integer less than or equal to $x$.

Analogous to Lemma \ref{le.2.1}, inserting (\ref{3.1}) into (\ref{1.2}),  with $\lambda_n=c-[a^2(n\pi)^2+b(n\pi)^{\alpha}]$,  we obtain
\begin{align*}
\mathrm{d}{U}_n(t)=\lambda_nU_n(t)\mathrm{d}t+\varepsilon q_nU_n(t)\mathrm{d}{W}_n(t),\quad  n=1,2,\ldots,
\end{align*}
with the initial value $U_n(0)=\langle U(0,\cdot),\tilde{e}_n(\cdot)\rangle$. Note that $U_n(0)=U_m(0)$ if $n=m$ and $U_n(0)=0$ if $n\neq m$. Therefore,
\begin{align}\label{3.2}
 {U}_n(t)=\exp\left\{b_nt+\varepsilon q_n{W}_n(t)\right\}U_n(0)=\left\{
         \begin{array}{ll}
\exp\{b_m t+\varepsilon q_m{W}_m(t)\}U_m(0), & n=m, \\ [0.2cm]
           0, & n\neq m,
         \end{array}
       \right.
\end{align}
where $b_n=\lambda_n-(\varepsilon q_n)^2/2$. In view of (\ref{3.1}) and (\ref{3.2}), it follows that
 \begin{align}\label{3.3}
 U(t,x)=U_m(t)\tilde{e}_m(x)=U_m(t,x).
 \end{align}

In the rest of this paper, let $\delta_0(u)$ denote the Dirac delta distribution centered at $0$ where $u\in\mathbb{R}$. Without loss of generality, we make the following assumption.

\smallskip
($\mathbf{A2}$) $U_m(0)$ obeys the logarithmic normal distribution with the density given by
\begin{align}\label{3.4}
p(u,0)=\left\{
         \begin{array}{ll}
           \frac{1}{u\sqrt{2\pi\nu_m(0)}}\exp\{-\frac{(\ln u-{\mu}_m(0))^2}{2\nu_m(0)}\}, & u>0, \\ [0.2cm]
         0, & \hbox{otherwise,}\\
         \end{array}
       \right.
\end{align}
where ${\mu}_m(0)=\mathbb{E}[\ln U_m(0)]\in\mathbb{R}$ and $\nu_m(0)={Var}[\ln U_m(0)]\in(0,\infty)$.

\begin{remark}
\rm
Note that $\tilde{e}_m(x)=\sqrt{2}\sin(m\pi x)$, $U_m(0)>0$ and $U( 0,x)=U_m(0)\tilde{e}_m(x)$. It follows that $U(0,x)=0$ if $x\in\{0,1/m,\ldots,(m-1)/m,1\}$, $U(0,x)>0$ if $x\in(2k/m,(2k+1)/m)$ and $U( 0,x)<0$ if $x\in((2k+1)/m,(2k+2)/m)\cap(0,1)$, where $k\in\{0,1,\ldots,\lfloor(m-1)/2\rfloor\}$. Therefore, if $(u,x)\in\mathbf{\Gamma}$, $U(0,x)$ obeys the Dirac delta distribution $\delta_0(u)$, $u\in\mathbb{R}$, which is independent of the choice of $U_m(0)$. Moreover, in view of $Var(\ln |U(0,x)|)={Var}[\ln U_m(0)]$ and the density function of $U_m(0)$ satisfying (\ref{3.4}), the density of $U( 0,x)$ satisfies
\begin{align}\label{3.5}
p(u,0,x)=\left\{
         \begin{array}{ll}
           \frac{1}{u\sqrt{2\pi\nu(0)}}\exp\{-\frac{(\ln u-{\mu}(0,x))^2}{2\nu(0)}\}, & (u,x)\in\mathbf{{D}}_{1,m}^k, \\ [0.2cm]
            \frac{1}{|u|\sqrt{2\pi\nu(0,x)}}\exp\{-\frac{(\ln |u|-{\mu}(0))^2}{2\nu(0)}\}, & (u,x)\in\mathbf{{D}}_{2,m}^k, \\ [0.2cm]
         0, & \mathbf{{\bar{D}}}\backslash(\mathbf{D}\cup\mathbf{\Gamma}),
         \end{array}
       \right.
\end{align}
where ${\mu}(0,x)=\mathbb{E}[\ln U_m(0)]+\ln |\tilde{e}_m(x)|$, $\nu(0)={Var}[\ln U_m(0)]$ .
\end{remark}
From the above remark, we know that the PDF $p(u,0,x)$ of the random initial value $U( 0,x)$ for $(\ref{1.2})$ depends on the spatial position $x$.
\begin{theorem}\label{th.3.2}
Let $(\mathbf{A2})$ hold and set $\varepsilon_m=\varepsilon q_m$,
\begin{align}\label{3.6}
 {\mu}(t,x)=\mathbb{E}[\ln U_m(0)]+b_mt+\ln |\tilde{e}_m(x)| \ {\rm and} \ \nu(t)={Var}[\ln U_m(0)]+\varepsilon_m^2 t.
\end{align}
Then the solution $U(t,x)$ for $(\ref{1.2})$ follows the Dirac delta distribution $\delta_0(u)$ if $(u,x)\in\mathbf{\Gamma}$, while it follows the logarithmic normal distribution with density given by
\begin{align}\label{3.7}
p(u,t,x)=\left\{
         \begin{array}{ll}
           \frac{1}{u\sqrt{2\pi\nu(t)}}\exp\{-\frac{(\ln u-{\mu}(t,x))^2}{2\nu(t)}\}, & (u,x)\in\mathbf{{D}}_{1,m}^k, \\ [0.2cm]
            \frac{1}{|u|\sqrt{2\pi\nu(t)}}\exp\{-\frac{(\ln |u|-{\mu}(t,x))^2}{2\nu(t)}\}, & (u,x)\in\mathbf{{D}}_{2,m}^k,\\ [0.2cm]
         0, & \mathbf{{\bar{D}}}\backslash(\mathbf{D}\cup\mathbf{\Gamma}),
         \end{array}
       \right.
\end{align}
where $k\in\{0,1,\ldots,\lfloor(m-1)/2\rfloor\}$. The density given by $(\ref{3.7})$ also satisfies
\begin{align}\label{3.8}
\lim\limits_{x\rightarrow\frac{k}{m}}p(u,t,x)=\delta_0(u), \quad u\in \mathbb{R},~t>0,~k=1,2,\ldots,m-1.
\end{align}
Moreover, if $(u,x)\in\mathbf{{D}}$, then $p(u,t,x)$ satisfies the following Fokker--Planck equation
\begin{align}\label{3.9}
\frac{\partial}{\partial t}p(u,t,x)=&\frac{\varepsilon_m^2}{2}u^2\frac{\partial^2}{\partial u^2}p(u,t ,x)+\left(\frac{3\varepsilon_m^2-b_m}{2}\right)u\frac{\partial}{\partial u}p(u,t ,x)+\left(\frac{\varepsilon_m^2-b_m}{2}\right) p(u,t ,x),
\end{align}
with the initial value $p(u,0,x)$ given in $(\ref{3.5})$. Furthermore, the following Chapman--Kolmogorov equation holds
\begin{align}\label{3.10}
\begin{array}{ll}
 p(u,t,x|w,s,x)=\int_{\mathbb{R}}p(u,t,x |v,r,x)p(v,r,x |w,s,x)dv
\end{array}
\end{align}
for $0<s<r<t$ and $x\in(0,1)\backslash\{0,1/m,\ldots,(m-1)/m,1\}$.
\end{theorem}
\begin{proof}
From (\ref{3.3}), we deduce that $U(t,x)=0$ for $x\in\{0,1/m,\ldots,(m-1)/m,1\}$. Therefore, $U(t,x)$ obeys the Dirac delta distribution $\delta_0(u)$ if $(u,x)\in\mathbf{\Gamma}$. When $(u,x)\in\mathbf{{\bar{D}}}\backslash\mathbf{\Gamma}$, we set $\hat{Y}_m(t,x)=\ln |U_m(t,x)|$. Thanks to $U_m(t,x)=U_m(t)\tilde{e}_m(x)$, $U_m(0)>0$ and (\ref{3.2}), $\hat{Y}_m(t,x)=\ln U_m(t)+\ln |\tilde{e}_m(x)|$. For every $t$, the random variable $\hat{Y}_m(t,x)\sim N(\mu(t,x),\nu(t))$ with $\mu(t,x)$ and $\nu(t)$ given by (\ref{3.6}). Note that $U(t,x)=U_m(t,x)$, the random variable $U(t,x)$ follows the logarithmic normal distribution, and the corresponding PDF $p(u,t,x)$ satisfies
\begin{align}\label{3.11}
p(u,t,x)=\left\{
         \begin{array}{ll}
           \frac{1}{u\sqrt{2\pi\nu(t)}}\exp\{-\frac{(\ln u-{\mu}(t,x))^2}{2\nu(t)}\}, & u>0,x\in(\frac{2k}{m},\frac{2k+1}{m}), \\ [0.2cm]
            0, & u\leq0,x\in(\frac{2k}{m},\frac{2k+1}{m}),\\ [0.2cm]
            \frac{1}{|u|\sqrt{2\pi\nu(t)}}\exp\{-\frac{(\ln |u|-{\mu}(t,x))^2}{2\nu(t)}\}, & u<0,x\in(\frac{2k+1}{m},\frac{2k+2}{m})\cap(0,1), \\ [0.2cm]
            0, & u\geq 0, x\in\left(\frac{2k+1}{m},\frac{2k+2}{m}\right)\cap(0,1)
         \end{array}
       \right.
\end{align}
with $k\in\{0,1,\ldots,\lfloor(m-1)/2\rfloor\}$. And the assertion (\ref{3.7}) follows  immediately. It remains to prove (\ref{3.8}) and we divide the proof into two cases.

$\mathbf{Case~1}$: $x\in(2k/m,(2k+1)/m)$, where $k\in\{0,1,\ldots,\lfloor(m-1)/2\rfloor\}$. Let $C_c^\infty(\mathbb{R})$ be the set of compactly supported infinitely differentiable functions. Take $\varphi\in C_c^\infty(\mathbb{R})$, we have
\begin{align*}
&\left\langle p(\cdot,t,x),\varphi(\cdot)\right\rangle=\int_0^\infty\frac{1}{u\sqrt{2\pi\nu(t)}}\exp\left\{-\frac{(\ln u-{\mu}(t,x))^2}{2\nu(t)}\right\}\varphi(u)du\\
&~~~~~~~~~~~~~~~~~~~=\frac{1}{\sqrt{2\pi\nu(t)}}\int_{-\infty}^\infty\exp\left\{-\frac{( u-{\mu}(t,x))^2}{2\nu(t)}\right\}\varphi(\exp(u))du\\
&~~~~~~~~~~~~~~~~~~~=\frac{1}{\sqrt{2\pi}}\int_{-\infty}^\infty\exp\left\{-\frac{r^2}{2}\right\}\varphi(\exp(\sqrt{\nu(t)}r+\mu(t,x)))dr=:J(t,x).
\end{align*}
By (\ref{3.6}) and the dominated convergence theorem,
\begin{align*}
\lim\limits_{x\rightarrow\frac{2k}{m}}J(t,x)=\frac{1}{\sqrt{2\pi}}\int_{-\infty}^\infty\exp\left\{-\frac{r^2}{2}\right\}dr\varphi(0)=\varphi(0)=\langle\delta_0(\cdot),\varphi(\cdot)\rangle,
\end{align*}
which implies that $\lim_{x\rightarrow\frac{2k}{m}}p(u,t,x)=\delta_0(u)$.

$\mathbf{Case~2}$: $x\in((2k+1)/m,(2k+2)/m)$, where $k\in\{0,1,\ldots,\lfloor(m-1)/2\rfloor\}$. By the similar arguments as in $\mathbf{Case~1}$, we obtain that $\lim_{x\rightarrow\frac{2k+1}{m}}p(u,t,x)=\delta_0(u)$. Therefore, the assertion (\ref{3.8}) holds.

Furthermore, for any $(u,x)\in\mathbf{{D}}$, the PDF $p(u,t ,x)$ satisfies
\begin{align}\label{3.12}
\left\{\begin{array}{ll}
\frac{\partial}{\partial t}p(u,t,x)=\frac{1}{2}\frac{\partial}{\partial t}{\nu}(t)\left[-\frac{1}{\nu(t)}+\frac{(\ln |u|-{\mu}(t,x))^2}{\nu^2(t)}\right]p(u,t ,x)
+\frac{\partial}{\partial t}{\mu}(t,x)\frac{\ln |u|-{\mu}(t,x)}{\nu(t)}p(u,t ,x), \\ [0.2cm]
\frac{\partial}{\partial u}p(u,t ,x)=-\frac{1}{u}\left(1+\frac{\ln |u|-{\mu}(t,x)}{\nu(t)}\right) p(u,t ,x),\\ [0.2cm]
\frac{\partial^2}{\partial u\partial u}p(u,t,x)=\frac{1}{u^2}\left\{\left[-3u\frac{\partial}{\partial u}p(u,t,x)-p(u,t ,x)\right]+\left[\frac{(\ln |u|-{\mu}(t,x))^2}{\nu^2(t)}-\frac{1}{\nu(t)}\right]p(u,t ,x)\right\}.
\end{array}\right.
\end{align}
Combining (\ref{3.6}) and (\ref{3.12}), the assertion (\ref{3.9}) holds. The last assertion (\ref{3.10}) can be obtained by the similar arguments as in Theorem \ref{th.2.4}.
\end{proof}
\begin{remark}
\rm
From (\ref{3.6}) and (\ref{3.7}), $U(t,x)$ obeys a logarithmic normal distribution with spatiotemporally dependent mean and purely time-dependent variance if $x\in(0,1)\backslash\big\{0,1/m,\ldots,(m-1)/m,1\}$. In addition, as $x$ approaches $k/m$, $U(t,x)$ tends to $0$ for any $t>0$. Furthermore, it follows from (\ref{3.8}) that when $x$ approaches $k/m$, the distribution of $U(t,x)$ transitions from the log-normal to the Dirac delta distribution that is independent of time $t$.
\end{remark}
\begin{remark}
\rm
Compared to the PDF of an SDE, the evolution of the PDF for an SPDE exhibits greater complexity. Theorem \ref{th.3.2} shows that the PDF of the solution $U(t,x)$ to (\ref{1.2}) critically depends on the spatial variable $x$. When $x\in\{0,1/m,\ldots,(m-1)/m,1\}$, $U$ follows a Dirac delta distribution, and when $x\in(0,1)\backslash\big\{0,1/m,\ldots,(m-1)/m,1\}$, $U$ obeys a logarithmic normal distribution. Moreover, as $x$ approaches $k/m$, the distribution of $U$ transitions from the log-normal to the Dirac delta distribution.
\end{remark}

Thanks to the Feynman--Kac formula \cite[Theorem 8.2.1]{BO}, we can establish the probability representation of the density function for the stochastic system (\ref{1.2}).
\begin{theorem}\label{th.3.5}
Suppose that $(u,x)\in\mathbf{{D}}$. Let $\hat{U}(t,x)$ be an It\^{o} process driven by
\begin{align}\label{3.13}
 d{\hat{U}( s,x)}=\left(\frac{3\varepsilon_m^2}{2}-b_m\right){\hat{U}( s,x)}{ds}+\varepsilon_m{\hat{U}( s,x)}dB(s),
  \end{align}
with the initial value $\hat{U}( 0,x)=u$, where $\varepsilon_m$ and $b_m$ are given in Theorem \ref{th.3.2}. Then the solution of the Fokker--Planck equation $(\ref{3.9})$ with the initial value $p(u,0 ,x)$ satisfying $(\ref{3.5})$ admits the probabilistic representation
\begin{align}\label{3.14}
p(u, t,x)=\exp\left\{\left(\frac{\varepsilon_m^2}{2}-b_m\right)t\right\}
\mathbb{E}\left[ p(\hat{U}(t,x), 0,x)|{\hat{U}( 0,x)=u}\right],\quad  t>0,~(u,x)\in\mathbf{{D}}.
\end{align}
\end{theorem}

\begin{remark}
\rm
For the following Stratonovich SHE with multiplicative $Q$-Wiener noise:
\begin{align}\label{3.15}
\begin{array}{clc}
 \frac{\partial}{\partial t}{U}(t,x)=a^2\frac{\partial^2}{\partial x^2}U(t,x)-b(-\frac{\partial^2}{\partial x^2})^{\frac{\alpha}{2}}U(t,x)+\left(c-\frac{\varepsilon^2}{2}\right)U(t,x)+\varepsilon \circ U(t,x)\frac{\partial}{\partial t}{W}(t,x),
\end{array}
\end{align}
where $W(t,x)$ is a $Q$-Wiener process with $q_m=1$, $t>0$ and $x\in(0,1)$, we can obtain results similar to Theorems \ref{th.3.2} and \ref{th.3.5} because its corresponding It\^{o} SHE is the first equation in (\ref{1.2}) with the same $Q$-Wiener noise. Furthermore, for the stochastic system $(\ref{1.2})$, due to the presence of the fractional-order term $b(-\frac{\partial^2}{\partial x^2})^{\frac{\alpha}{2}}U$, we are unable to analyze the PDF of the system under non-homogeneous boundary conditions. However, when $b=0$, we can still employ the methodology presented in Section $2$ to examine the PDF and the corresponding Fokker--Planck equation of the system $(\ref{1.2})$ under non-homogeneous boundary conditions.
\end{remark}
\begin{remark}
\rm
For the distribution of $\log U_m( 0)$, we can adopt the same distribution as that mentioned in Remark \ref{re.2.6}. And if $x$ belongs to a bounded domain $\Omega\subset\mathbb{R}^d$, the approach presented in this paper still suitable.
\end{remark}

\begin{figure}
\centering%
\includegraphics[height=2.5in,width=5in]{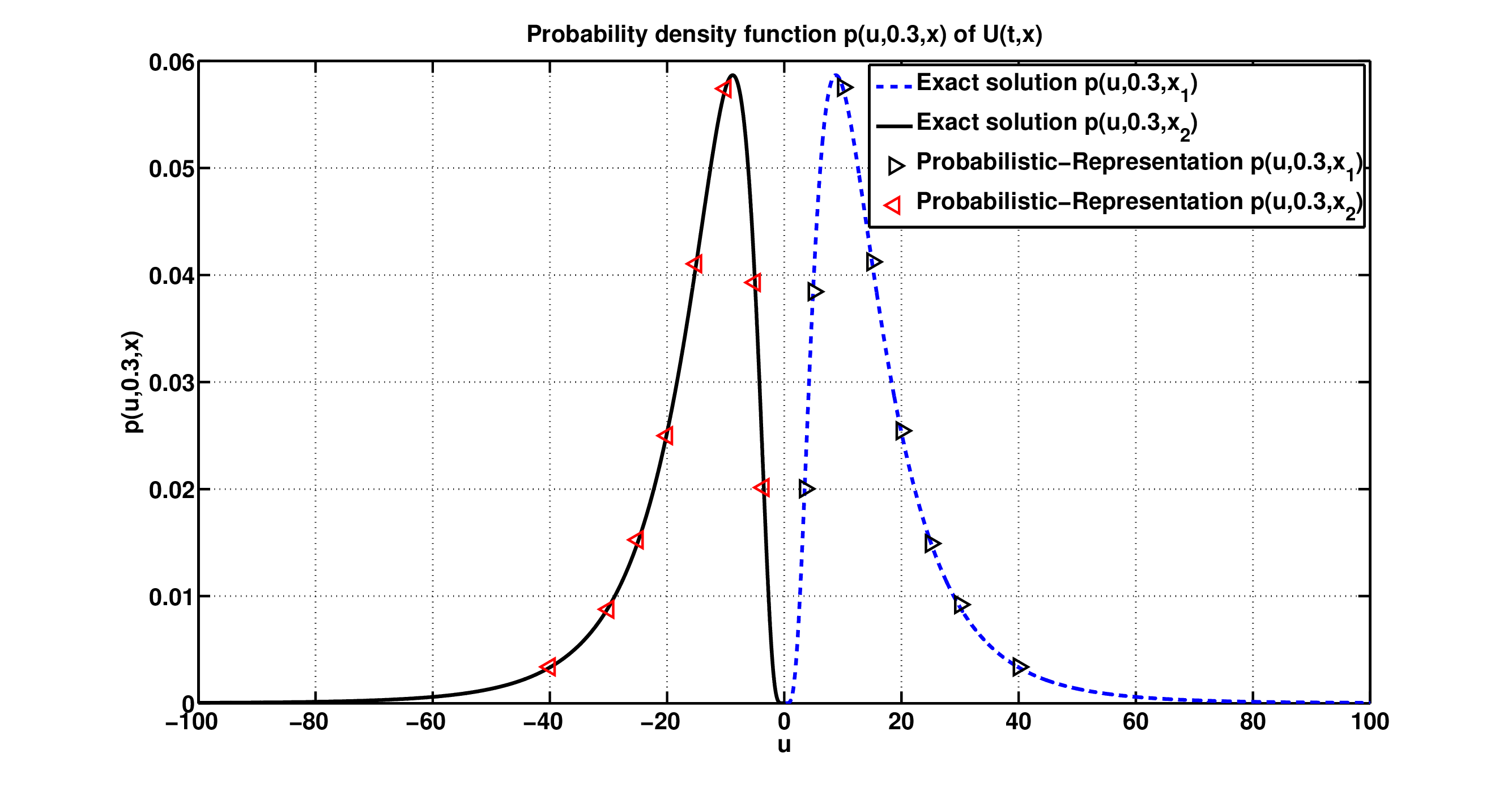}
\caption{Simulation of the PDF $p(u, 0.3,x)$ of (\ref{1.2}) with the step size $\triangle t=10^{-4}$ and different expressions.}
\label{2}
\end{figure}
\begin{figure}
\centering%
\includegraphics[height=2.5in,width=5in]{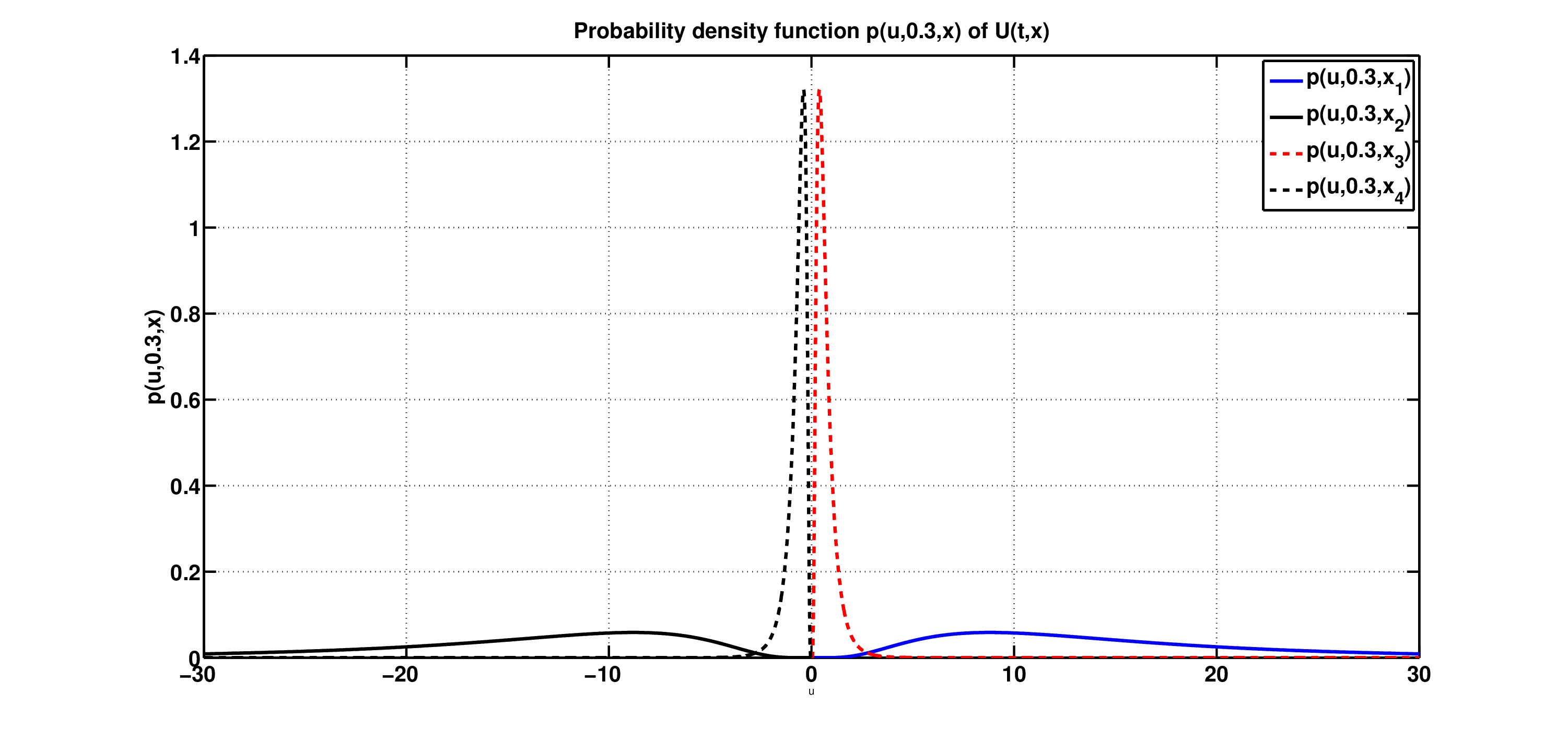}
\caption{Simulation of the PDF $p(u,0.3,x)$ of (\ref{1.2}) with the step size $\triangle t=10^{-4}$ and different $x$.}
\label{3}
\end{figure}
\begin{example}
\rm
Choose $a=1,b=1,\alpha=1/2,\varepsilon={\sqrt{2}}/2,m=2,q_m=1, c={11}/2+\left[\sqrt{2\pi}+(2\pi)^2\right],$ $x_1=1/8$, $x_2=5/8$, $x_3=99/200$ and $x_4=199/200$. Assume that $\ln U_m(0)$ obeys the normal distribution $N\left(1,{1}/4\right)$. In this situation, it follows from $\tilde{e}_m(x)=\sqrt{2}\sin(2\pi x)$ and (\ref{3.7}) that
\begin{align}\label{3.16}
p(u,t,x_i)=\left\{
         \begin{array}{ll}
           \frac{1}{u\sqrt{2\pi\nu(t)}}\exp\{-\frac{(\ln u-{\mu}(t,x_i))^2}{2\nu(t)}\}, & u>0,x_i\in(0,\frac{1}{2}), \\ [0.2cm]
            \frac{1}{|u|\sqrt{2\pi\nu(t)}}\exp\{-\frac{(\ln |u|-{\mu}(t,x_i))^2}{2\nu(t)}\}, & u<0,x_i\in(\frac{1}{2},1),\\ [0.2cm]
            0, & \hbox{otherwise,}
         \end{array}
       \right.
\end{align}
where ${\mu}(t,x_i)=(4+21t)/4+\ln |\sqrt{2}\sin(2\pi x_i)|$ and $\nu(t)=(1+2t)/4$, $i=1,2,3,4$.

In view of (\ref{3.14}), $p(u, t,x_i)$ also admits the following probabilistic representation
\begin{align}\label{3.17}
p(u, t,x_i)=\exp(-5t)\mathbb{E}\left[ p\left(\hat{U}\left( t,x_i\right), 0,x_i \right)\Big|{\hat{U}( 0,x_i)=u}\right],~~t>0,~(u,x_i)\in\mathbf{{D}}_{1,2}^0\cup\mathbf{{D}}_{2,2}^0,
\end{align}
where $p(u, 0,x_i)$ is given in (\ref{3.5}) and $\hat{U}\left( t,x_i\right)$ is an It\^{o} process driven by
\begin{align*}
 d{\hat{U}( t,x_i)}=-\frac{9}{2}{\hat{U}( t,x_i)}{dt}+\frac{\sqrt{2}}{2} {\hat{U}( t,x_i)}dB(t),
  \end{align*}
with the initial value $\hat{U}( 0,x_i)=u$ and $i=1,2$.

The computer simulations in Figure \ref{2} show the PDF $p(u, 0.3,x)$ of (\ref{1.2}) with diverse expressions (\ref{3.16})-(\ref{3.17}) and different spatial positions $x$. And from Figure \ref{3}, we find that as $x$ approaches $\frac{1}{2}$ or $1$, the distribution of $U$ transitions from the log-normal to the Dirac delta distribution.
\end{example}

\section{Kadar--Parisi--Zhang equation with an additive $Q$-Wiener noise}\label{sec4}

For the KPZ equation (\ref{1.3}), we first present the following result on the PDF and Fokker--Planck equation.

\begin{theorem}\label{th.4.1}
 Suppose that $(\mathbf{A2})$ holds. Let
 \begin{align}\label{4.1}
 {\tilde{\mu}}(t,x)=\frac{2\theta}{\xi}[\mathbb{E}[\ln U_m(0)]+\tilde{b}_mt+\ln |\tilde{e}_m(x)|]~\mathrm{and}~ \tilde{\nu}(t)=\frac{4\theta^2}{\xi^2}\left[{Var}[\ln U_m(0)]+(\varepsilon q_m)^2 t\right]
 \end{align}
 where $t\geq0$, $x\in(\underline{x},\overline{x})$ and $\tilde{b}_m=-\theta(m\pi)^2-(\varepsilon q_m)^2/2$. Then the PDF of the KPZ equation $(\ref{1.3})$ is
\begin{align}\label{4.2}
p(\kappa,t,x)=\frac{1}{\sqrt{2\pi\tilde{\nu}(t)}}\exp\left\{-\frac{(\kappa-{\tilde{\mu}}
(t,x))^2}{2\tilde{\nu}(t)}\right\}, \quad t\geq0,
\end{align}
and it satisfies the following Fokker--Planck equation
\begin{align}\label{4.3}
  \left\{
\begin{array}{ll}
\frac{\partial}{\partial t}p(\kappa,t,x)=-\frac{\partial{\tilde{\mu}}(t,x)}{\partial t}\frac{\partial}{\partial \kappa}p(\kappa,t,x)+\frac{1}{2}\frac{\partial\tilde{\nu}(t)}{\partial t}\frac{\partial^2}{\partial \kappa^2}p(\kappa,t,x),\\ [0.2cm]
p(\kappa,0,x)=\frac{1}{\sqrt{2\pi\tilde{\nu}(0)}}\exp\left\{-\frac{(\kappa-{\tilde{\mu}}(0,x))^2}{2\tilde{\nu}(0)}\right\}.
\end{array}
\right.
\end{align}
And the following Chapman--Kolmogorov equation holds
\begin{align}\label{4.4}
 p(\kappa,t,x|w,s,x)=\int_{\mathbb{R}}p(\kappa,t,x |v,r,x)p(v,r,x |w,s,x)dv ,\quad  0<s<r<t.
\end{align}
\end{theorem}
\begin{proof}
It follows from (\ref{3.2}), $a^2=\theta$, $b=c=0$ and $U_m(0)>0$ that
\begin{align}\label{4.5}
 {U}_m(t)=U_m(0)
\exp\{\tilde{b}_m t+\varepsilon q_m{W}_m(t)\}>0,
\end{align}
where $\tilde{b}_m=-\theta(m\pi)^2-(\varepsilon q_m)^2/2$. In view of the Cole--Hopf transformation $K(t,x)=2\theta\ln |U(t,x)|/\xi $ and (\ref{3.3}), $K(t,x)={2\theta}[\ln U_m(t)+\ln |\tilde{e}_m(x)| ]/\xi$. Thanks to (\ref{4.5}) and $W_m(t)$ is a standard scalar Brownian motion, the random variable ${K}(t,x)$ obeys the normal distribution $N\left({\tilde{\mu}}(t,x),\tilde{\nu}(t)\right)$
with ${\tilde{\mu}}(t,x)$ and $\tilde{\nu}(t)$ given by (\ref{4.1}), and the corresponding PDF $p(\kappa,t,x)$ satisfies (\ref{4.2}) and (\ref{4.3}). Due to $U(t,x)=U_m(t)\tilde{e}_m(x)$ is a Markov process and $K(t,x)=2\theta\ln |U(t,x)|/\xi $, the Chapman--Kolmogorov equation (\ref{4.4}) holds.
\end{proof}

By analogue  arguments as in Theorem \ref{th.2.7}, we obtain the following probabilistic representation  for $p(\kappa,t,x)$ of the solution process $K(t,x)$ of the KPZ equation $(\ref{1.3})$.
\begin{theorem}\label{th.4.2}
Under the assumption of Theorem \ref{th.4.1}. Suppose that $\hat{K}(s,x)$ is a solution of the following SDE
\begin{align}\label{4.6}
 d{\hat{K}( s,x)}=\frac{\partial{\tilde{\mu}}(T-s,x)}{\partial s}{ds}+\sqrt{-\frac{\partial\tilde{\nu}(T-s)}{\partial s}}dB(s),\quad T-t\leq s\leq T,
  \end{align}
with the initial condition $\hat{K}(T-t,x)=u$ and $x\in(\underline{x},\overline{x})\subset(k/m,(k+1)/m)$, $k=0,1,2,\ldots,m-1$. Then the density function $p$ of the KPZ equation  $(\ref{1.3})$ admits the probabilistic representation
\begin{align}\label{4.7}
p(\kappa, t,x)=\mathbb{E}\big[ p(\hat{K}(T,x),0,x)|{\hat{K}( T-{t},x)=\kappa}\big],\quad  (\kappa,t,x)\in\mathbb{R}\times[0,T]\times(\underline{x},\overline{x}).
\end{align}
\end{theorem}
\begin{remark}
\rm
For the following stochastic Burgers equation with an additive $Q$-Wiener noise:
\begin{align}\label{4.8}
\begin{array}{clc}
\frac{\partial}{\partial t}{H}(t,x)=\theta\frac{\partial^2}{\partial x^2}H(t,x)+\frac{\xi}{2}|\frac{\partial}{\partial x}H(t,x)|^2+\frac{2\theta\varepsilon}{\xi} \frac{\partial}{\partial t}{W}(t,x),
\end{array}
\end{align}
where $H(t,x)=\frac{\partial}{\partial x}K(t,x)$. Here $K(t,x)$ and $W(t,x)$ are given in (\ref{1.3}). We find that the PDF satisfied by the process corresponding to the integral of the solution to the above stochastic Burgers equation is the one obtained in the Theorems \ref{th.4.1} and \ref{th.4.2} since $K(t,x)=\int_a^xH(t,y)dy+K(t,a)$, $\underline{x}<a<x<\overline{x}$.
\end{remark}
\begin{figure}
\centering%
\includegraphics[height=2.5in,width=5in]{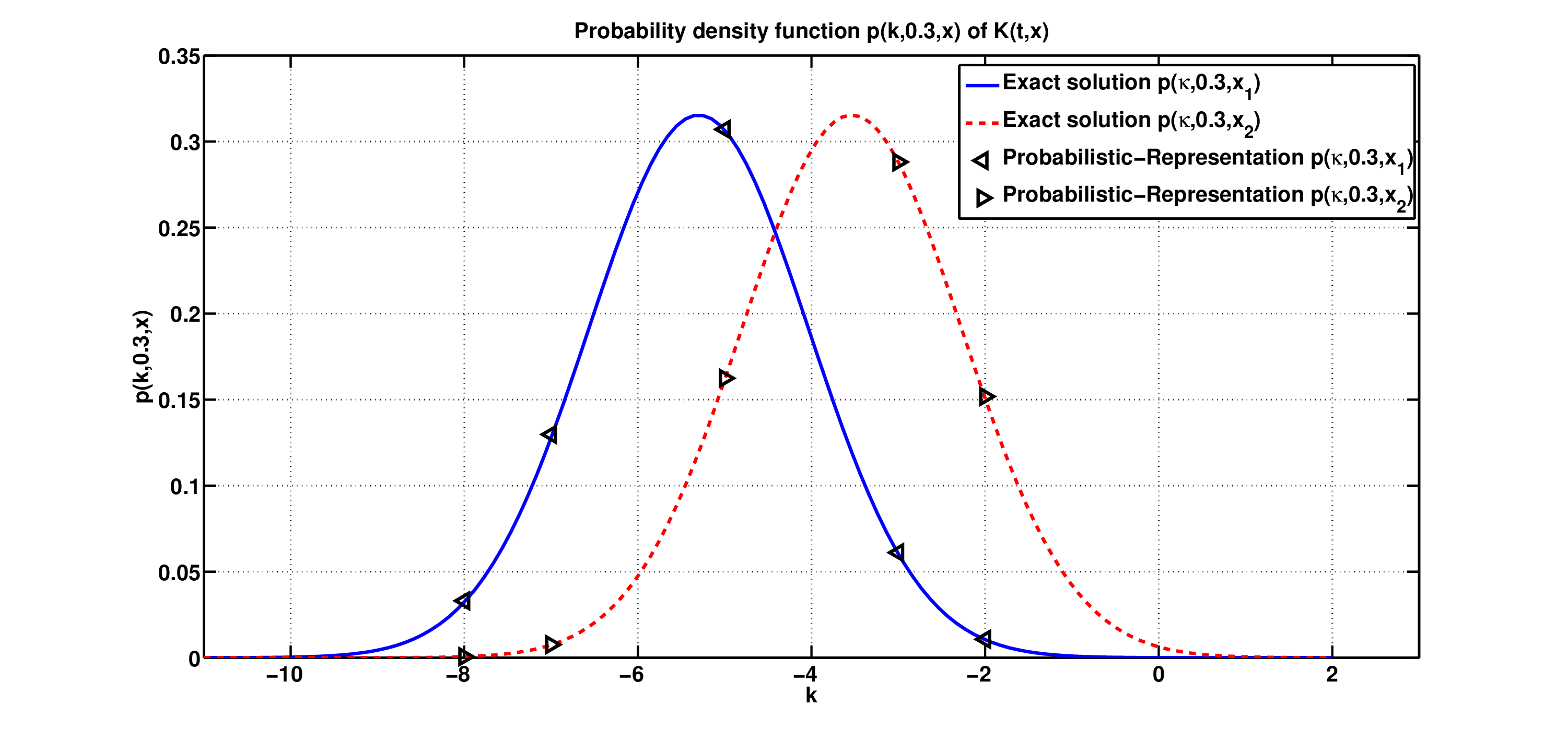}

\caption{Simulation of the PDF $p(\kappa, 0.3,x)$ of the KPZ equation (\ref{1.3}) with the step size $\triangle t=10^{-2}$, $T=0.5$, and different $\kappa$ and $x$.}
\label{4}
\end{figure}
\begin{example}
\rm
We adopt a set of parameters as follows $$\theta=1,\xi=1,\varepsilon={\sqrt{2}}/2,m=1,q_m=1, x_1=1/8~\mathrm{and}~x_2=5/8.$$
Let $\ln U_m(0)$ obey the normal distribution $N\left(1,{1}/4\right)$. Thanks to $\tilde{e}_m(x)=\sqrt{2}\sin(\pi x)$ and (\ref{4.1}), the density function $p(\kappa,t,x_i)$ of the KPZ equation $(\ref{1.3})$ is
\begin{align}\label{4.9}
p(\kappa,t,x_i)=\frac{1}{\sqrt{2\pi\tilde{\nu}(t)}}\exp\left\{-\frac{(\kappa-{\tilde{\mu}}
(t,x_i))^2}{2\tilde{\nu}(t)}\right\}, \quad t\geq0,
\end{align}
where ${\tilde{\mu}}(t,x_i)=2[1-(\pi^2+1/4)t+\ln |\sqrt{2}\sin(\pi x_i)|]$ and $\tilde{\nu}(t)=1+2t$, $i=1,2$.

In view of (\ref{4.7}), $p(\kappa, t,x_i)$ also admits the following probabilistic representation
\begin{align}\label{4.10}
p(\kappa, t,x_i)=\mathbb{E}\big[p(\hat{K}(T,x_i), 0,x_i)|{\hat{K}( T-t,x_i)=\kappa}\big],\quad (\kappa,t,x_i)\in\mathbb{R}\times[0,T]\times(\underline{x},\overline{x}),
\end{align}
where $p(\kappa, 0,x_i)$ is given in (\ref{4.9}) and $\hat{K}\left( t,x_i\right)$ is an It\^{o} process driven by
\begin{align*}
 \mathrm{d}{\hat{K}( t,x_i)}=2(\pi^2+1/4){\mathrm{d}t}+\sqrt{2}\mathrm{d}B(t),
  \end{align*}
with the initial value $\hat{K}(T-t,x_i)=\kappa$ and $i=1,2$.

By choosing $10^4$ samples of $\hat{K}(t,x)$, the computer simulations in Figure \ref{4} plot the PDFs $p(\kappa, 0.3,x)$ of (\ref{1.3}) with different $x$, and diverse expressions (\ref{4.9}) and (\ref{4.10}). It shows that the evolution of the PDF for the KPZ equation (\ref{1.3}) is related to the spatial position $x$.
\end{example}

\section{Conclusion}\label{sec5}

In this study, we focus on the derivation of the Fokker--Planck equation corresponding to a stochastic heat equation that incorporates additive $Q$-Wiener noise and is subject to non-homogeneous boundary conditions. A prominent technical challenge stems from the infinite-dimensional nature of (\ref{1.1}), which renders the common method in stochastic differential equations-deriving the adjoint operator of the generator through integration by parts-inapplicable. Furthermore, we extend our analysis to scenarios involving multiplicative noise with homogeneous boundary conditions and non-local diffusion operators. By employing sinusoidal initial excitation, we derive the Fokker--Planck equation corresponding to (\ref{1.2}), overcoming the difficulties introduced by multiplicative noise. Notably, we find that the evolution of the probability density function depends on the spatial locations $x$, and as $x$ approaches some specific discrete values $k/m$, the distribution transitions from the log-normal to the Dirac delta distribution. Finally, we provide the explicit probability density function and Fokker--Planck equation for the KPZ equations (\ref{1.3}) associated with equation (\ref{1.2}).

\bmhead{Acknowledgements}

This work was partly supported by the NSFC grant 12141107, the Guangdong Provincial Key Laboratory of Mathematical and Neural Dynamical Systems (Grant 2024B1212010004),   the Cross Disciplinary Research Team on Data Science and Intelligent Medicine (2023KCXTD054), the Guangdong-Dongguan Joint Research Grant 2023A1515140016, 2024A1515110193, and the China Postdoctoral Science Foundation under Grant 2025M773083.

\bmhead{Data availability} Data will be made available upon request.

\section*{Declarations}

\bmhead{Conflict of interest} The authors declare that there is no conflict of interest.

\end{document}